\documentclass[12pt]{amsart}
\usepackage[latin1]{inputenc}
\usepackage{amsmath}
\usepackage{amssymb}
\usepackage{amsrefs}
\usepackage{a4wide}
\usepackage{amscd} 
\usepackage{graphics}
\usepackage{color}
\usepackage{amsthm} 
\theoremstyle{plain}
\newtheorem{theorem}{Theorem}[section]
\newtheorem{lemma}[theorem]{Lemma}
\newtheorem{corollary}[theorem]{Corollary}
\newtheorem{proposition}[theorem]{Proposition}
\theoremstyle{definition}

\theoremstyle{remark}

\newtheorem*{thm}{Theorem}

\newcommand{\End}{\operatorname{End}}

\newcommand{\Rad}{\operatorname{Rad}}
\newcommand{\Jac}[1]{\mathrm{Jac}(#1)}

\newcommand{\QQ}{\mathbb{Q}}
\newcommand{\ZZ}{\mathbb{Z}}
\newcommand{\RR}{\mathbb{R}}
\newcommand{\C}{\mathcal{C}}
\begin{document}
\title{On a recent generalization of semiperfect rings}
\date{\today}
\author{Engin B\"uy\"uka\c{s}ik}
\address{Izmir Institute of Technology, Department of Mathematics, 35430, Urla, Izmir, Turkey}
\email{enginbuyukasik@iyte.edu.tr}

\author{Christian Lomp}
\address{Departamento de Matemâtica Pura da Faculdade de Ciências da Universidade do Porto, R.Campo Alegre 687, 4169-007 Porto, Portugal}
\email{clomp@fc.up.pt}

\thanks{The authors would like to thank Gena Puninski for pointing out
the example after Propositon \ref{semilocal_gs} and John Clark for having made valuable suggestions on our work. This paper was
written while the first author was visiting the University of Porto.
He wishes to thank the members of the Department of Mathematics for
their kind hospitality and the Scientific and Technical Research
Council of Turkey (T\"{U}BITAK) for their financial support. The
second author was supported by {\it Funda{\c{c}\~ao} para a
Ci\^encia e a Tecnologia} (FCT) through the {\it Centro de
Matem\'atica da Universidade do Porto} (CMUP)}


\begin{abstract}
It follows from a recent paper by Ding and Wang that any ring which
is generalized supplemented as left module over itself is semiperfect. The
purpose of this note is to show that Ding and Wang's claim is not true and that the class of generalized supplemented  rings lies properly
between the class of semilocal and semiperfect rings. Moreover we rectify their ``theorem" by introducing a wider notion of local
submodules.
 \end{abstract}
 \maketitle

\renewcommand{\theenumi}{\arabic{enumi}}
\renewcommand{\labelenumi}{\emph{(\theenumi)}}

\section{Introduction}

H.\ Bass characterized in \cite{Bass} those rings $R$ whose left
$R$-modules have projective covers and termed them {\it left perfect
rings}. He characterized them as those semilocal rings which have a
left $t$-nilpotent Jacobson radical $\Jac{R}$. Bass's {\it
semiperfect rings} are those whose finitely generated left (or
right) $R$-modules have projective covers and can be characterized
as those semilocal rings which have the property that idempotents
lift modulo $\Jac{R}$. 

Kasch and Mares transferred in
\cite{KaschMares} the notions of perfect and semiperfect rings to
modules and characterized semiperfect modules by a lattice-theoretical condition as follows. A module $M$ is called {\it
supplemented} if for any submodule $N$ of $M$ there exists a
submodule $L$ of $M$ minimal with respect to $M=N+L$. The left
perfect rings are then shown to be exactly those rings whose left
$R$-modules are supplemented while the semiperfect rings are those
whose finitely generated left $R$-modules are supplemented.
Equivalently it is enough for a ring $R$ to be semiperfect if the
left (or right) $R$-module $R$ is supplemented. Recall that a
submodule $N$ of a module $M$ is called {\it small}, denoted by
$N\ll M$, if $N+L\neq M$ for all proper submodules $L$ of $M$.
Weakening the ``supplemented'' condition one calls a module {\it
weakly supplemented} if for every submodule $N$ of $M$ there exists
a submodule $L$ of $M$ with $N+L=M$ and $N\cap L \ll M$. The
semilocal rings $R$ are precisely those rings whose finitely
generated left (or right) $R$-modules are weakly supplemented. Again
it is enough that $R$ is weakly supplemented as left (or right)
$R$-module. There are many semilocal rings which are not
semiperfect, e.g.\ the localization of the integers at two distinct
primes. 

Recently another notion of  ``supplement'' submodule has
emerged called {\it Rad-supplement}. A submodule $N$ of a module $M$
has a {\it generalized supplement} or {\it Rad-supplement} $L$ in
$M$ if $N+L=M$ and $N\cap L \subseteq \Rad(L)$ (see
\cite{generalized}). Here $\Rad(L)$ denotes the {\it radical} of $L$, namely the intersection of
all maximal submodules of $L$ or, if $L$ has no such submodules, simply $L$ itself. If every submodule of $M$ has a
Rad-supplement, then $M$ is called {\it Rad-supplemented}. Note that
Rad-supplements $L$ of $M$ are also called {\it coneat submodules}
in the literature and can be characterized by the fact that any
module with zero radical is injective with respect to the inclusion
$L\subseteq M$ (see \cite[10.14]{ALW} or \cite{ALW2}). Recall that a submodule $N$ of a module $M$ is called {\it cofinite} if $M/N$ is finitely generated.
In our terminology, Ding and Wang claimed that the following is true:

\begin{thm}[{Ding-Wang \cite[Theorem 2.11]{generalized}}]
 For any left $R$-module $M$ the following conditions are equivalent:
\begin{enumerate}
 \item[(a)] every submodule of $M$ is contained in a maximal submodule and every cofinite submodule of $M$ has a Rad-supplement in $M$.
\item[(b)] $M$ has a small radical and can be written as an irredundant sum of local submodules.
\end{enumerate}
 Here a module $L$ is called {\it local} if $L/\Rad(L)$ is simple and $\Rad(L)$ is small in $L$.
\end{thm}

Since it is well-known that finite sums of supplemented modules are
supplemented (see \cite[41.2]{Wisbauer}), Ding and Wang's claim implies  that a finitely generated module
$M$ is supplemented if and only if it is Rad-supplemented; in
particular for $M=R$, $R$ is semiperfect if and only if $R$ is
Rad-supplemented as left (or right) $R$-module. The purpose of this
note is to show that Ding and Wang's claim is not true and that
the class of Rad-supplemented rings lies properly between the class
of semilocal and semiperfect rings. Moreover we  rectify the
above theorem by introducing a different ``local'' condition for
modules.

Throughout the paper, $R$ is an associative ring with identity and
all modules are unital left $R$-modules.

\section{Examples}

It is clear from the introduction that one has the following
implications of conditions on submodules of a module with a small
radical:
$$\textit{supplements} \Rightarrow \textit{Rad-supplements} \Rightarrow \textit{weak supplements,}$$
and this implies that any semiperfect ring is Rad-supplemented and any
Rad-supplemented ring is semilocal.

\subsection{Semilocal rings which are not Rad-supplemented}
Since finitely generated modules have small radical it is clear that
finitely generated Rad-supplements are supplements. Thus any left
noetherian, left Rad-supplemented ring is semiperfect and any left
noetherian semilocal non-semiperfect ring would give an example of a
semilocal ring which is not Rad-supplemented. It is well-known that
the localization of the integers by two distinct primes $p$ and $q$, namely
$$ R= \left\{ \frac{a}{b} \in \QQ \mid b \text{ is neither divisible by }p \text{ nor by } q\right\},$$
is a semilocal noetherian ring which is not semiperfect and hence
gives such example.

\subsection{Rad-supplemented rings which are not semiperfect}

While a noetherian ring is Rad-supplemented if and only it
is semiperfect, we now show that a ring with idempotent Jacobson radical is semilocal
if and only if it is Rad-supplemented.

\begin{proposition}\label{semilocal_gs}
 Let $R$ be a ring with idempotent Jacobson radical $J$. Then $R$ is semilocal if and only if $R$ is Rad-supplemented as left or right $R$-module.
\end{proposition}
\begin{proof}
 Let $R$ be semilocal with idempotent Jacobson radical. Let $I$ be a left ideal $I$ of $R$. If $I\subseteq J$, then $I$ is small and $R$ is a Rad-supplement of $I$. Hence suppose that $I\not\subseteq J$. Then $(I+J)/J$ is a direct summand in the semisimple
 ring $R/J$ and there exists a left ideal $K$ of $R$ containing $J$ such that $I+K=R$ and $I\cap K \subseteq J$. Since $J\subseteq K$ and $J$  is
idempotent, we have $J=J^2\subseteq JK=\Rad(K)$. Thus $K$ is a
Rad-supplement of $I$ in $R$. 

The converse is clear.
\end{proof}

Hence any indecomposable non-local semilocal ring with idempotent
Jacobson radical is an example of a Rad-supplemented ring which is
not semiperfect. To construct such a ring we will look at the
endomorphism rings of uniserial modules. Recall that a unital ring
$R$ is called a {\it nearly simple uniserial domain} if it is
uniserial as left and right $R$-module, i.e.\ its lattices of left,
resp.\ right, ideals are linearly ordered, such that $\Jac{R}$ is the
unique nonzero two-sided ideal of $R$. Let $R$ be any nearly simple
uniserial domain and $r\in \Jac{R}$ and let $S=\End(R/rR)$ be the
endomorphism ring of the uniserial cyclic right $R$-module $R/rR$.
Then $S$ is semilocal by a result of Herbera and Shamsuddin (see
\cite[4.16]{Facchini}). Moreover Puninski showed in
\cite[6.7]{Puninski} that $S$ has exactly three nonzero proper
two-sided ideals, namely the two maximal ideals
$$I=\{f\in S \mid f \text{ is not injective}\}\:\:\mbox{ and }\:\:K=\{g\in S \mid g \text{ is not surjective}\},$$
and the Jacobson radical $\Jac S=I\cap K$ which is idempotent. Note
that since $R/rR$ is indecomposable, $S$ also has no non-trivial
idempotents. Moreover $S$ is a prime ring. Hence $S$ is an
indecomposable, non-local semilocal ring with idempotent Jacobson
radical.

\medskip

Our example depends on the existence of nearly simple uniserial
domains. Such rings were first constructed by Dubrovin in 1980 (see \cite{Dubrovin}), as
follows. Let $$G = \{ f:\QQ \rightarrow \QQ \mid f(t)=at+b \text{ for }  a,b\in
\QQ \text{ and }  a>0\}$$ be the group of affine linear functions on the
field of rational numbers $\QQ$. Choose any irrational number
$\epsilon \in \RR$ and set $$ P = \{ f \in G \mid \epsilon \leq
f(\epsilon) \} \:\: \mbox{ and } \:\: P^+=\{f\in G \mid \epsilon <
f(\epsilon)\}.$$ Note that $P$, resp.\ $P^+$, defines a left order on
$G$. Take an arbitrary field $F$ and consider the semigroup group
ring $F[P]$ in which the right ideal $M=\sum_{g \in P^+} gF[P]$ is
maximal. The set $F[P]\setminus M$ is a left and right Ore set and
the corresponding localization $R$ is a nearly simple uniserial
domain (see \cite[6.5]{BessenrodthBrungsTorner}). Thus taking any
nonzero element $r \in R$, $S=\End(R/rR)$ is a Rad-supplemented
ring which is not semiperfect.

%

\section{Cofinitely Rad-supplemented modules}

We say that the module $M$ is {\it cofinitely Rad-supplemented} if
any cofinite submodule of $M$ has a Rad-supplement. Note that every
submodule of a finitely generated module is cofinite, and hence a
finitely generated module is Rad-supplemented if and only if it is
cofinitely Rad-supplemented. Ding and Wang's theorem tried to
describe cofinitely Rad-supplemented modules with small radical as
sums of local modules. In order to correct their theorem we introduce
a different module-theoretic local condition.

\subsection{w-local modules}
We say that a module is {\it w-local} if it has a unique maximal submodule. It
is clear that a module is w-local if and only if its radical is
maximal.  The question whether any projective w-local module must be local was
first studied by R.\ Ware in \cite[p.\ 250]{Ware}. There he proved that
if $R$ is commutative or left noetherian or if idempotents lift modulo $\Jac{R}$, then any projective w-local module must be local
(see \cite[4.9--11]{Ware}). The first example of a w-local non-local projective module was given by Gerasimov and Sakhaev in \cite{GS}.

\medskip

In general, unless the module has small radical, a w-local module need not be local as we now show by taking the abelian group $M=\QQ\oplus \ZZ/p\ZZ$ for any prime $p$. Clearly $J = \QQ\oplus 0$ is maximal in $M$. We show that $J= \Rad(M)$. Indeed, if $N$ is another maximal submodule of $M$, then $M=N+J$ and thus $J/ (N\cap J) \simeq M/N$ is simple, which is impossible since $\QQ$ has no simple factors. Thus $J = \QQ\oplus 0$ is the unique maximal submodule of (the non-local module) $M$.

The same construction is possible for any ring $R$ with a nonzero left 
$R$-module $Q$ with $\Rad(Q)=Q$. Then taking any simple
left $R$-module $E$ one concludes that $Q\oplus E$ is w-local module
which is not local. In other words, if every w-local module over
a ring $R$ is local, then $\Rad(Q)\neq Q$ for all non-zero left
$R$-modules $Q$, i.e. $R$ is a {\it left max ring}. Conversely,
any left module over a left max ring has a small radical, hence any
w-local module is local. We have just proved:

\begin{lemma} A ring $R$ is a left max ring if and only if every $w$-local left $R$-module is local.
\end{lemma}

\subsection{Rad-supplements of maximal submodules}

While local (or hollow) modules are supplemented, w-local modules
might not be.

\begin{lemma}\label{Lemma: w-local is CGS} Every w-local module $M$ is cofinitely Rad-supplemented.
\end{lemma}

\begin{proof} Let $U$ be a cofinite submodule of $M$. Since $M/U$ is finitely generated, $U$ is contained in a maximal submodule of $M$ and so $U \subseteq \Rad M$. Now $M$ is a Rad-supplement of
$U$ in $M$, because $U+M=M$ and $U\cap M=U \subseteq \Rad M$.
\end{proof}

The gap in Ding and Wang's theorem is that Rad-supplements of
maximal submodules need not be local. However they are always w-local, as we now prove.

\begin{lemma}\label{Lemma: gs of maximal submodule is local}
Any Rad-supplement of a maximal submodule is w-local.
\end{lemma}

\begin{proof}
Let $K$ be a maximal submodule of a module $M$ and $L$ be a
Rad-supplement of $K$ in $M$. Then $K+L=M$ and $K\cap L \subseteq
\Rad L$. Since $M/K \cong L/K \cap L$ is simple, $\Rad L$ is a
maximal submodule of $L$. Thus $L$ is a $w$-local module.
\end{proof}

Let $R$ be again a nearly simple uniserial domain, $0\neq r\in
\Jac{R}$ and $S=\End(R/rR)$. Then $S$ has two maximal two-sided
ideals $I$ and $K$ such that $I\cap K=J=J^2$ where $J=\Jac{S}$. Since
$J\subset K$ we have $J=J^2 \subseteq JK \subseteq J$. Thus $I\cap
K=J=JK=\Rad(K)$. Analogously one shows $I\cap K = \Rad(I)$. Hence
$I$ and $K$ are mutual Rad-supplements. By \cite[p.\ 239 ]{Puninski}
$K$ is a cyclic uniserial left ideal, while the left ideal $I$ is not finitely
generated. Hence $K$ is a supplement of $I$. Moreover
$K$ is a supplement of any maximal left ideal $M$ of $S$ containing
$I$, because $K+M\supseteq K+I=S$ and $K\cap M\subseteq K$ is small
in $K$ since $K$ is uniserial. Hence $K$ is a local submodule. On
the other hand $I$ is not a supplement of $K$ since otherwise $K$
and $I$ would be mutual supplements and hence direct summands which is
impossible.


\subsection{Closure properties of cofinitely Rad-supplemented modules}

We establish some general closure properties of cofinitely
Rad-supplemented modules. To begin with, we prove the following.

\begin{lemma}\label{lemma:U+N has gs then U has ags}

Let $N\subseteq M$ be modules such that $N$ is cofinitely
Rad-supplemented. If $U$ is a cofinite submodule of $M$ such that $U+N$ has
a Rad-supplement in $M$, then $U$ also has a Rad-supplement in $M$.
\end{lemma}

\begin{proof} Let $K$ be a Rad-supplement of $U+N$ in $M$.
Since $N/[N\cap (U+K)] \cong M/(U+K)$ is finitely generated, $N\cap
(U+K)$ is a cofinite submodule of $N$. Let $L$ be a Rad-supplement
of $N\cap (U+K)$ in $N$, i.e.\ $L+(N\cap (U+K))=N$ and $L\cap
(U+K)\subseteq \Rad L.$ Then we have $$M=U+N+K=U+K+L+(N\cap
(U+K))=U+K+L$$ and $$U \cap (K+L)\subseteq (K \cap (U+L))+(L \cap
(U+K))\subseteq \Rad K + \Rad L \subseteq \Rad (K+L).$$ Hence $K+L$
is a Rad-supplement of $U$ in $M$.
\end{proof}

With this last Lemma at hand we can now prove:

\begin{theorem}\label{Corollary: Arbitrary sum of CGS}
The class of cofinitely Rad-supplemented modules is closed under
arbitrary direct sums and homomorphic images.
\end{theorem}

\begin{proof}
Let $\C$ denote the class of cofinitely Rad-supplemented left
$R$-modules over a fixed ring $R$. We will first show that $\C$ is
closed under factor modules. First of all, it is clear that $\C$ is
closed under isomorphisms, since being Rad-supplemented is a 
lattice-theoretical notion. Take any $M\in \C$ and $f:M\rightarrow X$, where
$X$ is a left $R$-module. We may assume that $f(M)$ is of the form
$M/N$ for some submodule $N$ of $M$. Let $K/N$ be a cofinite
submodule of $M/N$. Then $K$ is a cofinite submodule of $M$, so that
$K$ has a Rad-supplement $L$ in $M$. That is, $K+L=M$ and $K\cap L
\subseteq \Rad(L)$. Then we have $K/N + (L+N)/N=M/N$ and
$$K\cap(L+N)/N=(K\cap L +N)/N \subseteq (\Rad L +N)/N \subseteq \Rad
((L+N)/N).$$ Therefore $(L+N)/N$ is a Rad-supplement of $K/N$ in
$M/N$. Hence $M/N \in \C$.

Let $M=\bigoplus_{i \in I}M_{i}$, with $M_i \in \C$ and $U$ a
cofinite submodule of $M$. Then there is a finite subset
$F=\{i_1,i_2,\ldots, i_k\}\subseteq I$ such that $M=U + \left(
\bigoplus_{n=1}^k M_{i_n} \right)$. Since
$(U+\bigoplus_{n=2}^{k}M_{i_{n}})$ is cofinite and
$M=M_{i_{1}}+(U+\bigoplus_{n=2}^{k}M_{i_{n}})$ has trivially a
Rad-supplement in $M$, by Lemma \ref{lemma:U+N has gs then U has
ags} $U+\bigoplus_{n=2}^{k}M_{i_{n}}$ has a Rad-supplement in $M$.
By repeated use of Lemma \ref{lemma:U+N has gs then U has ags} $k-1$
times we get a Rad-supplement for $U$ in $M$. Hence $M$ is a
cofinitely Rad-supplemented.
\end{proof}

It follows from Theorem \ref{Corollary: Arbitrary sum of CGS} that a ring $R$ is left Rad-supplemented if and only if every left $R$-module is cofinitely Rad-supplemented.

Now we give an example showing that there are cofinitely Rad-supplemented modules, which are not Rad-supplemented. 
Let $N\subseteq M$ be left $R$-modules. If $M/N$ has no maximal submodule, then any proper Rad-supplement $L$ of $N$ in $M$ also has no maximal submodules, i.e.\ $L=\Rad(L)$, because if $L$ is a Rad-supplement of $N$ in $M$ and $K$ a maximal submodule of $L$, then 
$$M/(N+K) \simeq L/((N\cap L) + K) = L/K \neq 0$$
is a nonzero simple module,  a contradiction.

Next let $R$ be a discrete valuation ring with quotient field $K$. Then $F/T \cong K=\Rad(K)$ as $R$-modules, for a free $R$-module $F$ and $T \subseteq F$. Since $R$ is hereditary, any submodule of $F$ is projective and has a proper radical; thus $T$ cannot have a Rad-supplement in $F$ by the preceding remark, i.e.\ $F$ is not Rad-supplemented. On the other hand, since $R$ is local and hence Rad-supplemented, $F$ is cofinitely Rad-supplemented by Theorem \ref{Corollary: Arbitrary sum of CGS}.

\medskip

Call a module $M$ \emph{radical-full} if $M=\Rad(M)$. Obviously radical-full modules are Rad-supplemented.

\begin{lemma}\label{Lemma: M/N has no maximal and N CGS then M CGS} Any extension of a radical-full module by a cofinitely Rad-supplemented module is cofinitely Rad-supplemented.
\end{lemma}

\begin{proof}Let $N$ be a cofinitely Rad-supplemented submodule of $M$ such that $M/N$ is radical-full. For any cofinite $U\subseteq M$,   $(U+N)/N$ is a cofinite submodule of  $M/N$ and hence $U+N = M$. Hence, by Lemma \ref{lemma:U+N has gs then U has ags}, $U$ has a Rad-supplement in $M$.
\end{proof}

\subsection{Characterization of Rad-supplemented modules with small radical}

Recall that Ding and Wang's theorem expresses a cofinitely
Rad-supplemented module with small radical as an irredundant sum of
local modules, which is not possible in general as shown above. Here
we show that any cofinitely Rad-supplemented module can be written
as the sum of w-local modules instead. The following is a cofinitely Rad-supplemented analogue of a result  in \cite{ABS} by Alizade, Bilhan and Smith for
cofinitely supplemented modules.

\begin{theorem}\label{theorem:structure of CGS modules} The following statements are equivalent for a module $M$.
\begin{enumerate}
\item[(a)] $M$ is cofinitely Rad-supplemented;
\item[(b)] every maximal submodule of $M$ has a Rad-supplement in $M$;
\item[(c)] $M/wLoc(M)$ has no maximal submodules, where $wLoc(M)$ is the
sum of all w-local submodules of $M$;
\item[(d)] $M/cgs(M)$ has no maximal submodules, where $cgs(M)$ is the
sum of all cofinitely Rad-supplemented submodules of $M$.
\end{enumerate}
\end{theorem}

\begin{proof}$(a)\Rightarrow (b)$ Clear.

$(b)\Rightarrow (c)$ Suppose $wLoc(M) \subseteq K$ for some maximal
submodule $K$ of $M$. Let $N$ be a Rad-supplement of $K$ in $M$, i.e.\
$K+N=M$ and $K \cap N \subseteq \Rad N$. Then, by Lemma 
\ref{Lemma: gs of maximal submodule is local}, $N$ is a $w$-local
submodule of $M$ and so $N \subseteq wLoc(M) \subseteq K$, a contradiction.
Thus $M/wLoc(M)$ has no maximal submodules.

$(c)\Rightarrow (d)$ Suppose $K$ is a maximal submodule of $M$ with
$cgs(M) \subseteq K$. Since $M/wLoc(M)$ has no maximal submodules,
we have $K+wLoc(M)=M$. So $K+L=M$ for some $w$-local submodule $L$
of $M$. By Lemma \ref{Lemma: w-local is CGS}, $L$ is cofinitely
Rad-supplemented and so $L \subseteq cgs (M) \subseteq K$, a
contradiction. Hence $M/cgs(M)$ has no maximal submodules.

$(d)\Rightarrow (a)$ By Theorem \ref{Corollary: Arbitrary sum of
CGS}, $cgs(M)$ is cofinitely Rad-supplemented and so, by Lemma
\ref{Lemma: M/N has no maximal and N CGS then M CGS}, $M$ is
cofinitely Rad-supplemented.
\end{proof}

For finitely generated modules $M$, e.g.\ $M=R$,  we can rephrase the theorem as follows:

\begin{corollary}\label{Corollary: finitely gen GS} The following statements are equivalent for a finitely generated module $M$.
\begin{enumerate}
\item[(a)] $M$ is Rad-supplemented;
\item[(b)] every maximal submodule of $M$ has a Rad-supplement in $M$;
\item[(c)] $M$ is the sum of finitely many w-local submodules.
\end{enumerate}
\end{corollary}

\section{Concluding remarks}
While a ring is semiperfect (i.e. supplemented) if and only if it is
a sum of local submodules, we showed that a ring is left
Rad-supplemented if and only if it is a sum of w-local submodules.
Moreover the class of Rad-supplemented rings lies strictly between
the semilocal and the semiperfect rings. In particular any example
of a left Rad-supplemented ring which is not semiperfect, must
contain a w-local left ideal which is not local (cyclic).

\end{document}